\documentclass{amsart} 

\title{A metric between quasi-isometric trees}
\date{\today}

\author[Á.~Martínez-Pérez]{Álvaro Martínez-Pérez}
\address{Departamento de Geometría y Topología\\ Universidad Complutense de Madrid\\ Madrid 28040  Spain}
\email{alvaro\_martinez@mat.ucm.es}
       
\thanks{Partially supported by MTM 2009-07030.}

\usepackage[latin1]{inputenc}
\usepackage{amsfonts}
\usepackage{amsmath}
\usepackage{latexsym}
\usepackage[all]{xy}
\usepackage{amsthm}
\usepackage{graphicx}
\usepackage[T1]{fontenc}
\usepackage{hyperref}
\usepackage{caption}
\usepackage{verbatim}
\usepackage{amsmath}
\usepackage{amssymb}
\usepackage{amsthm}
\usepackage{amscd}
\usepackage{graphics}
\usepackage{graphpap}

\newtheorem{definicion}{Definition}[section]
\newtheorem{nota}[definicion]{Remark}
\newtheorem{prop}[definicion]{Proposition}
\newtheorem{lema}[definicion]{Lemma}
\newtheorem{obs}[definicion]{Remark}
\newtheorem{teorema}[definicion]{Theorem}
\newtheorem{cor}[definicion]{Corollary}
\newtheorem{ejp}[definicion]{Example}

\newtheorem{notation}[definicion]{Notation}

\newcommand{\br}{\ensuremath{\mathbb{R}}} 
\newcommand{\co}{\ensuremath{\colon}} 
\newcommand{\bn}{\ensuremath{\mathbb{N}}} 
\newcommand{\e}{\ensuremath{\mathrm{e}}} 

\begin{document}

\begin{abstract} It is known that PQ-symmetric maps on the boundary characterize the quasi-isometry type of visual hyperbolic spaces, in particular, of geodesically complete \br-trees. We define a map on pairs of PQ-symmetric ultrametric spaces which characterizes the branching of the space. We also show that, when the ultrametric spaces are the corresponding end spaces, this map defines a metric between rooted geodesically complete simplicial trees with minimal vertex degree 3 in the same quasi-isometry class. Moreover, this metric measures how far are the trees from being rooted isometric.
\end{abstract}

\maketitle

\tableofcontents

\begin{footnotesize}
\end{footnotesize}

\begin{footnotesize}
\end{footnotesize}

\section{Introduction}

The study of  quasi-isometries between trees and 
the induced maps on their end spaces has a  voluminous literature.
This is often set in the more general context of hyperbolic metric spaces
and their boundaries.
See Bonk and Schramm \cite{BoSchr},
Buyalo and Schroeder \cite{BS},
Ghys and de la Harpe \cite{G-H},
and Paulin \cite{Pau}
to name a few.

For a quasi-isometry $f\co X\to Y$ between Gromov hyperbolic, almost geodesic metric spaces, Bonk and Schramm 
define \cite[Proposition 6.3]{BoSchr} 
the induced map $\partial f\co\partial X\to\partial Y$ between the boundaries at infinity
and prove \cite[Theorem 6.5]{BoSchr} 
that $\partial f$ is  PQ-symmetric  
with respect to any metrics on $\partial X$ and $\partial Y$ 
in their canonical gauges. Moreover, they prove that a PQ-symmetric map between bounded metric spaces can be extended to a map between their hyperbolic cones and obtain that a PQ-symmetric map between the boundaries at infinity of Gromov hyperbolic, almost geodesic metric spaces, implies a quasi-isometry equivalence between the spaces.
In the special case that $X$ and $Y$ are $\br$-trees, $\partial X= end(X,v)$, $\partial Y=end(Y,w)$ and the end space metrics
are in the canonical gauges for any choice of roots.

Another source for the result that quasi-isometries between $\br$-trees induce PQ-symmetric homeomorphisms on their ultrametric end spaces,
is Buyalo and Schroeder \cite[Theorem 5.2.17]{BS}. They work with Gromov hyperbolic, geodesic metric spaces and with visual boundaries on their
boundaries. When specialized to $\br$-trees, these boundaries are the ultrametric end spaces. For the converse in this approach see \cite{M}.

In \cite{Hug-M-M} we defined bounded distortion equivalences and we proved that bounded distortion equivalences characterize power PQ-symmetric homeomorphisms
between certain classes of bounded, complete, uniformly perfect, ultrametric spaces which we called pseudo-doubling.  This class includes those ultrametric spaces arising up to similarity as the end spaces of bushy trees. 

Bounded distortion equivalences can be seen from a geometrical point of view as a coarse version of quasiconformal homeomorphisms (see for example \cite{Ahl}, \cite{Hei} and \cite{Hub} for a geometric approach to quasiconformal maps) where instead of looking at the distortion of the spheres with the radius tending to 0 we consider the distortion of all of them. 

Of central importance in the theory of quasiconformal mappings are Teichmüller spaces. The Teichmüller space is the set of Riemannian surfaces of a given quasiconformal type and the Teichmüller metric measures how far are the spaces from being conformal equivalent. There is also an extensive literature on this, see for example \cite{Ahl} and \cite{Hub}. The question we deal with in this paper is to see if something similar to a Teichmüller metric can be defined with these bounded distortion equivalences playing the role of the quasiconformal homeomorphisms.

Here we consider the set of ultrametric spaces of a given PQ-symmetry type. What is obtained is not a metric in the general framework of pseudo-doubling ultrametric spaces because it fails to hold the triangle inequality, nevertheless, adding a condition on the metrics, it is enough to characterize what we call here the \emph{branching} of the space which is a natural concept when looking at the ultrametric space as boundary of a tree. 

\begin{teorema} Let $(U,d),(U',d')$ be ultrametric spaces. If the metrics $d,d'$ are pseudo-discrete, then $\varrho(U,U')=0$ if and only if $(U,d)$ and $(U',d')$ have the same branching.
\end{teorema}

As we mentioned before, if the ultrametric spaces are end spaces of \br--trees, then the PQ-symmetry type of the boundary corresponds to the quasi-isometry type of the trees.
In the case in which the trees have minimal vertex degree 3, the defined map $\varrho$ is a metric measuring how far are the trees from being rooted isometric.

\begin{teorema} $\varrho$ is an unbounded metric in $[(R,z)]$ such that given $(T,v),(T',w)\in [(R,z)]$, $\varrho((T,v),(T',w))=0$ if and only if there is a rooted isometry between $(T,v)$ and $(T',w)$. 
\end{teorema}

This means that in the category $T_{\geq 3}$ the branching is enough to characterize isometry type, the quasi-isometry type and the ``distance'' between quasi-isometric objects.

\section{Preliminaries on trees, end spaces, and ultrametrics}

In this section,  we recall the definitions of the trees and their end spaces that are
relevant to this paper. 
We also describe a well-known correspondence between trees and ultrametric
spaces. See Fe{\u\i}nberg \cite{Feinberg} for an early result along these lines and  Hughes \cite{Hug} for additional background.

\begin{definicion} Let $(T,d)$ be a metric space. 
\begin{enumerate}
	\item $(T,d)$ is an \emph{$\mathbb{R}$--tree} if $T$ is uniquely arcwise connected and for all $x, y \in T$, the
unique arc from $x$ to $y$, denoted $[x,y]$, is isometric to the
subinterval $[0,d(x,y)]$ of $\mathbb{R}$.

\item
A \emph{rooted $\mathbb{R}$--tree} $(T,v)$ consists of an $\br$-tree $(T,d)$ and a point $v\in T$, called  the \emph{root}.

\item
\label{extensiongeod} A rooted $\mathbb{R}$--tree $(T,v)$ is
\emph{geodesically complete} if every isometric embedding
$f\co [0,t]\rightarrow T$ with $t>0$ and  $f(0)=v$ extends to an isometric embedding $F\co [0,\infty) \rightarrow T$. 

\item A \emph{simplicial $\mathbb{R}$--tree} is an $\br$-tree $(T,d)$ such that $T$ is the (geometric realization of) a simplicial complex
and every edge of $T$ is isometric to the closed unit interval $[0,1]$.
\end{enumerate}
\end{definicion}

\begin{notation} For any $x\in (T,v)$, let $||x||=d(x,v)$.
\end{notation}

\begin{definicion} An \emph{ultrametric space} is a metric space $(X,d)$ such that 
$d(x,y)\leq \max \{d(x,z),d(z,y)\}$
for all $x,y,z\in X$. 
\end{definicion}

\begin{definicion}\label{end} The \emph{end space} of a rooted
$\mathbb{R}$--tree $(T,v)$ is given by: 
$$end(T,v)=\{F\co [0,\infty) \rightarrow T \ |\ \text{$F(0)=v$ and $F$
is an isometric embedding}\}.$$ 

Let $F, G\in end(T,v)$.
\begin{enumerate}
	\item  The {\it Gromov product at infinity} is $(F|G)_v :=\sup \{t\geq 0 \ |\ F(t)=G(t)\}$.
	\item  The {\it end space metric} is $d_v(F,G) := \e^{-(F|G)_v}$.
	\item The arc $F([0, (F|G)_v])$ is denoted $[F|G]$.
\end{enumerate}
\end{definicion}

\begin{prop} If $(T,v)$ is a rooted $\mathbb{R}$--tree,
then $(end(T,v),d_v)$ is a complete ultrametric space of diameter
$\leq 1$.\qed
\end{prop}

In this article a \emph{map} is a function that need not be continuous.

\begin{definicion}
\label{def:quasi-isometry}
A map $f\co X\to Y$ between metric spaces $(X,d_X)$ and $(Y,d_Y)$ is
a \emph{quasi-isometric map} if there are constants $\lambda
\geq 1$ and $A>0$ such that for all $x,x'\in X$,
$$\frac{1}{\lambda}d_X(x,x') -A \leq d_Y(f(x),f(x'))\leq \lambda
d_X(x,x')+A.$$ If $f(X)$ is a net in $Y$ (i.e., there exists $\epsilon >0$ such that for each $y\in Y$ there exists $x\in X$ such that
$d_Y(f(x), y) <\epsilon$), then $f$ is a
\emph{quasi-isometry} and $X,Y$ are \emph{quasi-isometric}.
\end{definicion}

\begin{definicion} A map $f:X \to Y$ between metric spaces is
called \emph{quasi-symmetric} if it is not constant and if there
is a homeomorphism $\eta:[0,\infty) \to [0,\infty)$ such that from
$|xa|\leq t|xb|$ it follows that $|f(x)f(a)|\leq
\eta(t)|f(x)f(b)|$ for any $a,b,x\in X$ and all $t\geq 0$. 
The function $\eta$ is called the \emph{control function} of $f$.
\end{definicion}

\begin{definicion}\label{def:PQ-symmetric} A quasi-symmetric map is said to be \emph{power quasi-symmetric}
or \emph{PQ-symmetric}, if its control function is of the form
\[\eta(t)= q \max\{t^p,t^{\frac{1}{p}}\}\] for some $p,q\geq 1$.
\end{definicion}

\section{Bounded distortion equivalences between ultrametric spaces.}

Let us recall some definitions as stated in \cite{Hug-M-M}.

\begin{definicion}
Let $f\co X\to Y$ be a homeomorphism between metric spaces
$(X,d_X)$ and $(Y,d_Y)$. If
$x_0\in X$ and $\varepsilon >0$, then the {\it distortion by $f$ of
the $\varepsilon$-sphere $S(x_0,\varepsilon) := \{ x\in X ~|~ d_X(x_0,x) =\varepsilon\}$ at $x_0$} is
$$D_f(x_0,\varepsilon) :=
\begin{cases}
\frac{\sup\{ d_Y(f(x_0),f(x)) ~|~ d_X(x_0,x)=\varepsilon\}}{\inf\{d_Y(f(x_0),f(x)) ~|~ d_X(x_0,x)=\varepsilon\}} & \text{if $S(x_0,\varepsilon)\not=\emptyset$}\\
\quad\quad\quad\quad\quad\quad 1 & \text{if $S(x_0,\varepsilon)=\emptyset$}
\end{cases}$$
\end{definicion}

\begin{definicion}
\label{def:qc et al}
Let $f\co X\to Y$ be a homeomorphism between  metric spaces. 
\begin{enumerate}
	\item $f$ is {\it conformal} if $\underset{\varepsilon\to 0}\limsup D_f(x_0,\varepsilon) =1$ for all $x_0\in X$.
		\item $f$ is {\it $K$-quasiconformal}, where $K>0$, if $\underset{\varepsilon\to 0}\limsup D_f(x,\varepsilon) \leq K$ for all $x\in X$.
			\item $f$ is {\it quasiconformal} if $f$ is $K$-quasiconformal for some $K>0$.
			\item $f$ has {\it bounded distortion} if there exists $K>0$ such that
$$\underset{x\in X}\sup\, \underset{\varepsilon> 0}\sup\, D_f(x,\varepsilon) \leq K.$$
\item $f$ is a {\it bounded distortion equivalence} if $f$ and $f^{-1}\co Y\to X$ have bounded distortion.
\end{enumerate}
\end{definicion}

Consider $(U,d),(U',d')$ two bounded distortion equivalent ultrametric spaces. Let $\mathcal{K}_{U,U'}$ be the greatest lower bound for $K$ such that there exist a homeomorphism $f\co U \to U'$ with  $\underset{x\in U}\sup\, \underset{\varepsilon> 0}\sup\, D_f(x,\varepsilon) \leq K$ and $\underset{x'\in U'}\sup\, \underset{\varepsilon> 0}\sup\, D_{f^{-1}}(x',\varepsilon) \leq K$. Then, given $\mathcal{U}$ a bounded distortion equivalence class of ultrametric spaces, let us define $\varrho\co \mathcal{U}\times \mathcal{U}\to \br_+ $ such that $$\varrho(U,U'):=ln(1+ 2ln(\mathcal{K}_{U,U'})).$$

\begin{nota} It could be interesting to extend the previous statement defining $\varrho(U,U'):=\infty$ if  $U$ and $U'$ are homeomorphic but not bounded distortion equivalent.
\end{nota}

\begin{definicion}\label{def:ramif} Given a bijection $h\co (U,d)\to (U',d')$ between ultrametric spaces we say that $h$ \textbf{preserves the branching}  
if given three points $x,y,z\in U$, $d(x,y)=d(x,z)$ implies that $d'(h(x),h(y))=d'(h(x),h(z))$ and $d(x,y)<d(x,z)$ implies that $d'(h(x),h(y))<d'(h(x),h(z))$ (i.e. $d(x,y)=d(x,z)$ if and only if $d'(h(x),h(y))=d'(h(x),h(z))$). If there exists such a bijection we say that the ultrametric spaces have the \textbf{same branching}.
\end{definicion}

\begin{obs} Note that this defines an equivalence relation.
\end{obs}

\begin{definicion} Let $(X,d)$ be a metric space. We say that $d$ is \textbf{pseudo-discrete} if there is some $\delta>1$ such that for every non-empty sphere $S(x,r)$, and any $y$ such that $\frac{r}{\delta}<d(y,x)<r\cdot \delta$ then $d(x,y)=r$.
\end{definicion}

\begin{teorema}\label{prop:branching} Let $(U,d),(U',d')$ be ultrametric spaces. If the metrics $d,d'$ are pseudo-discrete, then $\varrho(U,U')=0$ if and only if $(U,d)$ and $(U',d')$ have the same branching.
\end{teorema}

\begin{proof} If the branching is the same then there is a homeomorphism $h$ such that the spheres are preserved, this is, $D_h(x,r)=1$ and $D_{h^{-1}(h(x),r)}=1$ for any $x\in U$ and any $r>0$.

Suppose that both ultrametric spaces are pseudo-discrete with the same constant $\delta>1$ and let $\varrho(U,U')=0$ which is equivalent to saying that $\mathcal{K}_{U,U'}=1$. Then, for any $\varepsilon >1$ there exists a homeomorphism $h_\varepsilon \co U \to U'$ such that given $x,y,z \in U$ with $d(x,y)=d(x,z)$  then, assuming $d'(h_\varepsilon(x),h_\varepsilon(y))\leq d'(h_\varepsilon(x),h_\varepsilon(z))_w$, 
$$\frac{d'(h_\varepsilon(x),h_\varepsilon(z))}{d'(h_\varepsilon(x),h_\varepsilon(y))}<\varepsilon.$$ Taking $\varepsilon <\delta$, we conclude that $d'(h_\varepsilon(x),h_\varepsilon(y))=d'(h_\varepsilon(x),h_\varepsilon(z))$.

Now suppose that $d(x,y)<d(x,z)$. Then $d(y,z)_v=d(x,z)$ by the properties of the ultrametric and, as we just proved,  taking $\epsilon <\delta$, $d'(h_\varepsilon(y),h_\varepsilon(z))=d'(h_\varepsilon(x),h_\varepsilon(z))$ and therefore 
$d'(h_\varepsilon(x),h_\varepsilon(y)) \leq d'(h_\varepsilon(x),h_\varepsilon(z))$. 

If $d'(h_\varepsilon(x),h_\varepsilon(z)) =d'(h_\varepsilon(x),h_\varepsilon(H))$ the same argument on $h^{-1}$ would imply that $d(x,y)=d(x,z)$ leading to contradiction and we obtain that $d'(h_\varepsilon(x),h_\varepsilon(y)) < d'(h_\varepsilon(x),h_\varepsilon(y))$ finishing the proof.  
\end{proof}

Let us recall the following definition from \cite{Hug-M-M}

\begin{definicion} A metric space is  \emph{pseudo-doubling} if for every $C>1$ there exist $N\in \bn$ such that:
if $0< r< R$ with $R/r=C$ and $x\in X$, then there are at most $N$ balls $B$ such that $B(x,r)\subseteq B \subseteq B(x,R)$.
\end{definicion}

\begin{nota} It is immediate to check that if $(X,d)$ is a metric space and $d$ is pseudo-discrete, then $(X,d)$ is pseudo-doubling.
\end{nota}

Nevertheless, two pseudo-doubling ultrametric spaces $U,U'$ with $\varrho(U,U')=0$ need not have the same branching.

\begin{ejp}\label{ejp1}
Let $(T,v),(T',w)$ be rooted geodesically complete \br--trees with: $end(T,v)=\{F_n,G_n,H_n \ | \ n=0,1,2,...\}$ and $end(T',w)=\{F'_n,G'_n,H'_n \ | \ n=1,2,...\}$ and the following relations.

For $end(T,v)$:

\begin{itemize} \item $(F_i|F_j)_v=0 \ \forall \, i\neq j$.
\item $(F_0|G_0)_v= (F_0|H_0)_v=1$.
\item $(F_i|G_i)_v=1 \ \forall \, i\geq 1$ 
\item $(F_i|H_i)_v=2 \ \forall \, i\geq 1$.
\end{itemize}

For $end(T',w)$:

\begin{itemize} \item $(F'_i|F'_j)_w=0 \ \forall \, i\neq j$.
\item $(F'_i|G'_i)_v=1 \ \forall \, i\geq 1$ 
\item $(F'_i|H'_i)_v=1+\frac{1}{i} \ \forall \, i\geq 1$.
\end{itemize}

Let us define $h_n\co end(T,v)\to end(T',w)$ such that $h_n(F_0)=F'_n$, $h_n(G_0)=G'_n$, $h_n(H_0)=H'_n$, $h_n(F_i)=F'_i$, $h_n(G_i)=G'_i$, $h_n(H_i)=H'_i \ \forall \, 1\leq i<n$ and $h_n(F_i)=F'_{i+1}$, $h_n(G_i)=G'_{i+1}$ and $h_n(H_i)=H'_{i+1} \ \forall \, i\geq n$. 

It is immediate to check that there is no sphere distorted by $h_n^{-1}$ and the unique spheres distorted by $h_n$ are $S_{1}(F_0)$ and $S_1(H_0)$ in $end(T,v)$ where $D_{h_n}(F_0,1)=e^{\frac{1}{n}}= D_{h_n}(G_0,1)$. Therefore, for any $K>1$, there exists some $n\in \bn$ such that $e^{\frac{1}{n}}<K$ and $\mathcal{K}_{(T,v),(T',w)}=1$ which implies that $\varrho ((T,v),(T',w))=0$.

Clearly, the ramification is not the same. From the cardinality together with the bounded distortion condition we know that $h(S_1(F_i))$ corresponds to $S_1(h(F'_i))$ for any $i$, and in $(T,v)$ we have $(F_0|G_0)_v=(F_0|H_0)_v=(G_0|H_0)_v$ while for any bijection $h\co end(T,v) \to end(T',w)$ either $(h(F_0)|h(G_0))_w \neq (h(F_0)|h(H_0))_w$ or $(h(F_0)|h(G_0))_w \neq (h(G_0)|h(H_0))_w$.\qed
\end{ejp}


\section{A metric between quasi-isometric trees}

Consider $(T,v),(T',w)$ two rooted geodesically complete \br--trees such that $end(T,v)$ and $end(T',w)$ are bounded distortion equivalent. When we are considering rooted trees we will denote $\mathcal{K}_{end(T,v),end(T',w)}$ just by $\mathcal{K}_{(T,v),(T',w)}$ and, therefore, we can define $$\varrho((T,v),(T',w)):=ln(1+ 2ln(\mathcal{K}_{(T,v),(T',w)})).$$

\begin{nota} We may also consider $\varrho((T,v),(T',w)):=\infty$ if  $end(T,v)$ and $end(T',w)$ are homeomorphic but not bounded distortion equivalent.
\end{nota}

Now, translating \ref{def:ramif} to trees:

\begin{obs} Let $(T,v),(T',w)$ be rooted geodesically complete \br--trees and $h\co end(T,v) \to end(T',w)$ be a bijection. Then, $h$ preserves the branching if given three branches $F,G,H\in end(T,v)$, if $(F|G)_v=(F|H)_v$ then $(h(F)|h(G))_v=(h(F)|h(H))_v$ and if $(F|G)_v<(F|H)_v$ then $(h(F)|h(G))_v<(h(F)|h(H))_v$. If there exist such a bijection, we say that $(T,v)$ and $(T',w)$ have the same branching.
\end{obs}

Now, from Proposition \ref{prop:branching}, we have:

\begin{cor} Let $(T,v),(T',w)$ be rooted geodesically complete \br--trees. If the metrics $d_v,d_w$ of $end(T,v),end(T,w)$ are pseudo-discrete, then $\varrho((T,v),(T',w))=0$ if and only if the ramification of $(T,v)$ and $(T',w)$ is the same.
\end{cor}

In, particular, the end space metric of a rooted geodesically complete simplicial tree is pseudo-discrete with $\delta=e$.

\begin{cor} If $(T,v),(T',w)$ are rooted simplicial geodesically complete \br--trees, then $\varrho((T,v),(T',w))=0$ if and only if the ramification of $(T,v)$ and $(T',w)$ is the same.
\end{cor}

This is not true in general for rooted geodesically complete \br--trees as we saw in Example \ref{ejp1}.










\begin{definicion} An $\br$-tree $T$ is \emph{bushy} if there is
a constant $K>0$, called a {\it bushy constant}, such that for any point $x\in T$ there is
a point $y\in T$ such that $d(x,y)< K$ and $T\backslash \{y\}$ has at least
$3$ unbounded components.
\end{definicion}

The following result is from \cite{Hug-M-M}.

\begin{teorema}
\label{cor:trees} 
A homeomorphism $h\co end(T,v)\to end(T',w)$ between the end spaces
of rooted, geodesically complete, simplicial, bushy $\mathbb{R}$--trees is
PQ-symmetric if and only if $h$ is
a bounded distortion equivalence.
\end{teorema}

A map $f\co (T,v) \to (T',w)$ is a rooted isometry if it is an isometry and $f(v)=w$. If such a map exists, 
we say that $(T,v)$ and $(T',w)$ are rooted isometric, which defines an equivalence relation.

Let us consider, from now on, the category $\mathcal{T}$ of rooted isometry classes of rooted geodesically complete simplicial bushy trees. 

Given $(S,x)\in \mathcal{T}$, let $[(S,x)]$ be the class of rooted geodesically complete \br--trees quasi-isometric to $(S,x)$ (i.e. whose end space is bounded distortion equivalent to $end(S,x)$).

The extended version, allowing the image of $\varrho$ to be $\infty$ together with Theorem \ref{cor:trees}, implies the following.

\begin{prop} $\varrho((T,v),(T',w))<\infty$ if and only if $end(T,v)$ and $end(T',w)$ are PQ-symmetric. \qed
\end{prop}

Let us recall that, from \cite{BoSchr} (see also \cite{M}), as a particular case we have that

\begin{prop} $end(T,v)$ and $end(T',w))$ are PQ-symmetric if and only if $(T,v)$ and $(T',w)$ are quasi-isometric.\qed
\end{prop}

Let us consider the restriction of the category $\mathcal{T}$ to rooted geodesically complete simplicial trees with valence at least 3 at each vertex, $\mathcal{T}_{\geq 3}$. Let $[(R,z)]\subset \mathcal{T}_{\geq 3}$ be the class of trees quasi-isometric to $(R,z)$ (i.e. whose end space is bounded distortion equivalent to $end(R,z)$).

\begin{prop}\label{isom} Given $(T,v),(T',w)\in [(R,z)]$, $\varrho((T,v),(T',w))=0$ if and only if there is an isometry between $end(T,v)$ and $end(T',w)$. 
\end{prop}

\begin{proof} The if part is obvious. 

Suppose that $\varrho((T,v),(T',w))=0$. Since the trees are simplicial $D_g(F,r)\in \{e^{k}\, | \, k=0,1,... \}$ for any homeomorphism $g$ and any radius $r>0$. Therefore, there is a homeomorphism $h\co end(T,v)\to end(T',w)$ such that $D_h(F,r)=1$ for any $F\in end(T,v)$ and any $r>0$. Let us prove that $h$ is an isometry by induction on the Gromov product. 

Let $F,G\in end(T,v)$ such that $(F|G)_v=0$, i.e., $d_v(F,G)=1$ and suppose $(h(F)|h(G))_w\geq 1$. Since the minimal vertex degree is 3, there is a point $H'\in end(T',v)$ such that $(h(F)|H')_w=0$ and $(H'|h(G))_w=0$. Since $D_{h^{-1}}(H',1)=1$, necessarily $(F|h^{-1}(H'))_v=0$ and $(h^{-1}(H')|G)_v=0$ but then $D_h(F,1)\geq e$ leading to contradiction. The same works for $h^{-1}$.

Now suppose that for every $k<n$ and every pair of branches $F_1,G_1\in end(T,v)$ (resp. $F'_1,G'_1\in end(T',w)$) such that $(F_1|G_1)_v=k$ (resp. $(F'_1|G'_1)_w=k$) then $(h(F_1)|h(G_1))_w=k$ (and $(h^{-1}(F'_1)|h^{-1}(G'_1))_w=k$). Let $F_2,G_2\in end(T,v)$ such that $(F_2|G_2)_v=n$ and suppose that $(h(F_2)|h(G_2))_w\neq n$. If $(h(F_2)|h(G_2))_w< n$ then, the assumption on $h^{-1}$ would imply that $(F_2|G_2)_v<n$. Hence $(h(F_2)|h(G_2))_w>n$ and since $T'$ has minimal vertex degree 3, there exist some $H'$ with $(h(F_2)|H')_w=n=(H'|h(G_2))_w$. It is immediate to check that $(F_2|h^{-1}(H'))_v=n=(h^{-1}(H')|G_2)_v$ and this leads to a contradiction with $D_h(F_2,e^{-n})=1$.
\end{proof}

Let us recall the following corollary in
\cite{Hug}.

\begin{prop}\label{cor_H} Two geodesically complete rooted $\mathbb{R}$--trees,
$(T,v)$ and $(S,w)$, are rooted isometric if and only if
$end(T,v)$ and $end(S,w)$ are isometric.
\end{prop}

From Propositions \ref{isom} and \ref{cor_H} we obtain:

\begin{cor}\label{cor:isometry} Given $(T,v),(T',w)\in [(R,z)]$, $\varrho((T,v),(T',w))=0$ if and only if there is a rooted isometry between $(T,v)$ and $(T',w)$. 
\end{cor}

\begin{nota} Note that this distance depends on the ramification of the tree and not on its isometry type. Therefore, the root plays an important role and it is immediate to check that from the same tree with two different roots $(T,v)$, $(T,w)$ it may happen that $\varrho((T,v),(T,w))>0$. Nevertheless, it can be bounded above by a constant depending on the distance between the roots since $\varrho((T,v),(T,w))\leq ln(1+d(v,w))$.
\end{nota}

\begin{prop}\label{metric} $\varrho$ is a metric in $[(R,z)]$.
\end{prop}

\begin{proof} $\varrho$ is real valued and non negative by definition. By \ref{cor:isometry}, $\varrho((T,v),(T',w))=0$ if and only if $(T,v)=(T',w)$.

The symmetric property follows immediately from the definition. Hence, it suffices to check the triangle inequality.

Let $(T,v), (T',w),(T'',x)\in [(R,z)]$ with $$\varrho((T,v),(T',w))=d_1=ln(1+2ln(\mathcal{K}_{(T,v),(T',w)}))$$ and $$\varrho((T',w),(T'',x))=d_2=ln(1+2ln(\mathcal{K}_{(T',w),(T'',x)})).$$ To simplify the notation through the proof let us denote $\mathcal{K}_{(T,v),(T',w)}$ as $\mathcal{K}_1$ and $\mathcal{K}_{(T',w),(T'',x)}$ as $\mathcal{K}_2$.

Since $\mathcal{K}_1$ and $\mathcal{K}_2$ are greatest lower bounds, there is a homeomorphism $h_1\co end(T,v) \to end(T'w)$ such that $$\underset{F\in end(T,v)}\sup\, \underset{\varepsilon> 0}\sup\, D_{h_1}(F,\varepsilon) \leq \frac{e}{2}\mathcal{K}_1 \mbox{ and } \underset{F'\in end(T',w)}\sup\, \underset{\varepsilon> 0}\sup\, D_{h_1^{-1}}(F',\varepsilon) \leq \frac{e}{2}\mathcal{K}_1,$$ and a homeomorphism $h_2\co end(T',w) \to end(T'',x)$ with the corrosponding condition for $\mathcal{K}_2$. Also, since the trees are simplicial, for any homeomorphism $h$, $D_{h}(F,\varepsilon)$ takes values in $\{e^n \ | \ n\in \bn\}$. Hence, $D_{h_1}(F,\varepsilon)\leq \mathcal{K}_1$, $D_{h_1^{-1}}(F',\varepsilon) \leq \mathcal{K}_1$ and the same happens with $h_2$.

Case 1. $\mathcal{K}_1=1$. Then $\varrho((T,v),(T'',x))=\varrho((T,v),(T',w))+\varrho((T',w),(T'',x))=\varrho((T',w),(T'',x))$. It suffices to observe that for any $F\in end(T,v)$ and $G,H\in S(F,\varepsilon)$, there exists some $\varepsilon'>0$ such that $h_1(G),h_1(H)\in S(h_1(F),\varepsilon')$ and, assuming $d_x(h_2(h_1(F)),h_2(h_1(H)))\leq d_x(h_2(h_1(F)),h_2(h_1(G)))$, then $$\frac{d_x(h_2(h_1(F)),h_2(h_1(G)))}{d_x(h_2(h_1(F)),h_2(h_1(H)))}\leq \mathcal{K}_2.$$   

Case 2. If $\mathcal{K}_1>1$ then $\mathcal{K}_1\geq e$ and $ln\mathcal{K}_1\geq 1$. 

Suppose $F\in end(T,v)$ and $G,H\in S(F,\varepsilon)$ and let us assume that $$t_0=(h_1(G)|h_1(F))_w\leq (h_1(H)|h_1(F))_w=t_1.$$ Then, since $D_{h_1}(F,\varepsilon)\leq \mathcal{K}_{(T,v),(T',w)}$, $$t_1-t_0\leq ln(\mathcal{K}_1).$$

Claim: $$|(h_2(h_1)(F)|h_2(h_1(H)))_x-(h_2(h_1)(F)|h_2(h_1(G)))_x| \leq ln(\mathcal{K}_1)\cdot ln(\mathcal{K}_2) +2 ln(\mathcal{K}_2)+1.$$

If $(h_2(h_1)(F)|h_2(h_1(H)))_x < (h_2(h_1)(F)|h_2(h_1(G)))_x$, then, by the bounded distortion condition with respect to $h_1(G)$ we obtain that $(h_2(h_1)(F)|h_2(h_1(G)))_x-(h_2(h_1)(F)|h_2(h_1(H)))_x \leq ln(\mathcal{K}_2)$ proving the claim. Hence, we will assume that $(h_2(h_1)(F)|h_2(h_1(H)))_x \geq (h_2(h_1)(F)|h_2(h_1(G)))_x$.

Let $\mathcal{F}_i$ be the set of branches $F'\in end(T',w)$ such that $(h(F)|F')_w=t_0+i$ with $i=0,t_1-t_0$. For any $F'_1,F'_2 \in \mathcal{F}_i$, the bounded distortion condition with respect to $h_1(F)$ implies that $$|(h_2(h_1(F))|h_2(F'_1))_x-(h_2(h_1(F))|h_2(F'_2))_x|\leq ln(\mathcal{K}_2).$$ Thus, if \begin{equation}\label{bounded}(h_2(h_1)(F)|h_2(h_1(H)))_x-(h_2(h_1)(F)|h_2(h_1(G)))_x > (t_1-t_0+1)ln(\mathcal{K}_2)+C,\end{equation} since the vertices in $T''$ have valence at least 3, there are $C$ branches $F_j''$, $j=1,C$, in $(T'',x)$ such that $(h_2(h_1(F))|h_2(h_1(G))_x<(h_2(h_1(F))|F_j'')_x<(h_2(h_1(F))|h_2(h_1(H))_x$ which are not in $h_2(\mathcal{F}_i)$ for any $i=0,t_1-t_0$ and such that $(h_2(h_1(F))|F_j'')_x\neq (h_2(h_1(F))|F_i'')_x$ for any $j\neq i$. Suppose $C=ln(\mathcal{K}_2)+1$.

In this condition, either $(h_2(h_1)(F)|h_2(h_1(H)))_x -(h_2(h_1)(F)|h_2(h_1(G)))_x\leq 2C$, which would prove the claim, or there is some $F''_{j_0}$ such that $(h_2(h_1)(F)|h_2(h_1(H)))_x -(h_2(h_1)(F)|F''_{j_0})_x\geq C$ and $(h_2(h_1)(F)|F''_{j_0})_x -(h_2(h_1)(F)|h_2(h_1(G)))_x\geq C$ with $h_2^{-1}(F''_{j_0})\not \in \mathcal{F}_i$ for any $i$, which means that either $(h_1(F),h_2^{-1}(F''_{j_0}))_w<t_0$ or $(h_1(F),h_2^{-1}(F''_{j_0}))_w>t_1$. In both cases, the bounded distortion condition is not hold and (\ref{bounded}) leads to a contradiction.

From this we conclude that $(h_2(h_1)(F)|h_2(h_1(H)))_x-(h_2(h_1)(F)|h_2(h_1(G)))_x \leq (t_1-t_0+1)ln(\mathcal{K}_2)+ ln(\mathcal{K}_2)+1\leq (ln(\mathcal{K}_1)+1)ln(\mathcal{K}_2)+ln(\mathcal{K}_2)+1$ proving the claim.

Thus, $$ln(\mathcal{K}_3)=ln(\mathcal{K}_{(T,v),(T'',x)})\leq ln(\mathcal{K}_1)\cdot ln(\mathcal{K}_2) +2 ln(\mathcal{K}_2)+1.$$ Since $\varrho((T,v),(T'',x))=ln(1+2ln(\mathcal{K}_3))$ it follows that \begin{equation}\label{triang}\varrho((T,v),(T'',x))\leq ln(1+2ln(\mathcal{K}_1)\cdot ln(\mathcal{K}_2) +4 ln(\mathcal{K}_2)+2).\end{equation}

Now, observe that $\varrho((T,v),(T',w))+\varrho((T',w),(T'',x))=ln(1+2ln(\mathcal{K}_1))+ln(1+2ln(\mathcal{K}_2))=ln(4 ln(\mathcal{K}_1)\cdot ln(\mathcal{K}_2)+2ln(\mathcal{K}_1)+2ln(\mathcal{K}_2)+1)$.

Since $ln(\mathcal{K}_1)\geq 1$, from (\ref{triang}) we obtain that $$\varrho((T,v),(T'',x))\leq ln(1+2ln(\mathcal{K}_1)\cdot ln(\mathcal{K}_2)+ 2ln(\mathcal{K}_1)\cdot ln(\mathcal{K}_2) +2 ln(\mathcal{K}_2)+2ln(\mathcal{K}_1))=$$
$$=\varrho((T,v),(T',w))+\varrho((T',w),(T'',x)).$$
\end{proof}

\begin{obs} Given $(T,v),(T',w)\in [(R,z)]$, $\varrho((T,v),(T',w))\in \{ln(1+2n) \ | \ n=1,2...\}$.
\end{obs}

\begin{obs} $([(R,z)],\varrho)$ is discrete.
\end{obs}

\begin{nota}\label{unique tree} As Mosher, Sageev, and Whyte \cite[page 118]{Mo} point out,
any two bounded valence (i.e. there is a constant $A$ such that the valence of 
every vertex is less than $A$), locally finite, simplicial, bushy trees
are quasi-isometric. 
Therefore, any such tree is quasi-isometric to the infinite binary tree.
See also Bridson and Haefliger \cite[page 141, Exercise 8.20(2)]{B-H} for the special case of
regular, simplicial trees. 
\end{nota}

Let $\mathcal{T}_{\geq 3}^*$ the class of bounded valence trees in $\mathcal{T}_{\geq 3}$. From \ref{unique tree} it follows:

\begin{prop} $(\mathcal{T}_{\geq 3}^*,\varrho)$ is a metric space.
\end{prop}

Given a rooted simplicial tree $(T,v)$ and a vertex $x \in T$, let us denote the order of $x$ by $ord(x)$ and let us call the vertices whose geodesic to the root $v$ contains $x$, \textbf{descendants of $x$}.  If $k\in \bn$, by $desc_k(x)$ we denote the set of descentants of $x$ at a distance $k$ whith the canonical metric on the tree (edges having length 1).

Given $x\in (T,v)$, we denote by $T_x=\{y\in T \ | \ x\in [v,y]\}$ is also a tree. If $(T,v)$ is geodesically complete, so it is $(T_x,x)$ and there is a canonical injection $j\co end(T_x,x) \to end(T,v)$. Abusing of the notation we may consider $end(T_x,x) \subset end(T,v)$.

\begin{lema} Let $(T,v)$, $(T',w)$ be two rooted simplicial trees and $x$ be a vertex in $T$. If $ord(x)=m+1$ and $\mathcal{D}=\min_k \{sup_{x'\in \mathcal{V}(T')}\{desc_k(x')\}<m\}$ then $\varrho((T,v),(T'w))\geq ln(1+2\mathcal{D})$.
\end{lema}

\begin{proof} Consider any bijection $h\co end(T,v) \to end(T'w)$. Let $x_1, ..., x_m$ be $m$ different vertices in $desc_1(x)$ and let $\mathcal{F}_i=end(T_{x_i},x_i) \subset end(T,v)$ for $i=1,m$. There is a unique vertex $y\in T'$ such that $h(F)\in end(T'_y,y)$ and $|wy|$ is maximal. Then, there are three branches $F_{i_1},F_{i_2},F_{i_3}$ such that $[h(F_{i_1})|h(F_{i_2})]=[w,y]$, $[h(F_{i_1})|h(F_{i_3})]=[w,y]$ and $(h(F_{i_2})|h(F_{i_3}))_w\geq |wy|+\mathcal{D}$. Therefore, $D_h(F_1,e^{-||x||}) \geq \mathcal{D}$ and $\varrho((T,v),(T'w))\geq ln(1+2\mathcal{D})$.
\end{proof}

Therefore, $(\mathcal{T}_{\geq 3}^*,\varrho)$ is an unbounded metric space.

\end{document}